\documentclass[12pt]{amsart}

\usepackage{amssymb}
\usepackage{bm}
\usepackage{graphicx}
\usepackage{psfrag}
\usepackage{enumerate}
\usepackage{url}
\usepackage{comment}
\usepackage{algpseudocode}
\usepackage{mathtools}
\usepackage{blkarray}
\usepackage{hyperref}
\allowdisplaybreaks

\newcommand{\excise}[1]{}

\newtheorem{thm}{Theorem}[section]
\newtheorem{lemma}[thm]{Lemma}

\newtheorem{cor}[thm]{Corollary}
\newtheorem{prop}[thm]{Proposition}
\newtheorem{conj}[thm]{Conjecture}

\theoremstyle{definition}

\newtheorem{example}[thm]{Example}
\newtheorem{remark}[thm]{Remark}
\newtheorem{defn}[thm]{Definition}

\newtheorem{prob}[thm]{Problem}

\numberwithin{equation}{section}

\renewcommand\>{\rangle}
\newcommand\<{\langle}

\newcommand\RR{\mathbb{R}}

\newcommand\ZZ{\mathbb{Z}}

\newcommand\kk{\Bbbk}

\newcommand\xx{{\mathbf x}}

\renewcommand\aa{{\mathbf a}}
\newcommand\bb{{\mathbf b}}

\newcommand\nothing{\varnothing}

\DeclareMathOperator\image{im} 
\DeclareMathOperator\sign{sign} 
\DeclareMathOperator\rank{rank} 
\DeclareMathOperator\In{In} 
\DeclareMathOperator\Tor{Tor} 
\DeclareMathOperator\lcm{lcm} 
\DeclareMathOperator\Ap{Ap} 

\newcommand\s{\scriptstyle}
\newcommand\phm{\phantom{-}}
\newcommand\rlm{-}

\newcommand\filleftmap{\mathord\leftarrow \mkern-6mu
	\cleaders\hbox{$\mkern-2mu \mathord- \mkern-2mu$}\hfill
	\mkern-6mu \mathord-}

\begin{document}

\title[Minimal free resolutions of numerical semigroup algebras]%
{Minimal free resolutions of numerical\\semigroup algebras via Ap\'ery specialization}

\author[B.~Braun]{Benjamin Braun}
\address{Mathematics Department\\University of Kentucky\\Lexington, KY 40506}
\email{benjamin.braun@uky.edu}

\author[T.~Gomes]{Tara Gomes}
\address{School of Mathematics\\University of Minnesota\\Minneapolis, MN 55455}
\email{gomes072@umn.edu}

\author[E.~Miller]{Ezra Miller}
\address{Department of Mathematics\\Duke University\\Durham, NC 27708}
\urladdr{\url{https://math.duke.edu/people/ezra-miller}}

\author[C.~O'Neill]{Christopher O'Neill}
\address{Mathematics Department\\San Diego State University\\San Diego, CA 92182}
\email{cdoneill@sdsu.edu}

\author[A.~Sobieska]{Aleksandra Sobieska}
\address{Mathematics Department\\University of Wisconsin Madison\\Madison, WI 53706}
\email{asobieska@wisc.edu}

\makeatletter
  \@namedef{subjclassname@2020}{\textup{2020} Mathematics Subject Classification}
\makeatother
\subjclass[2020]{Primary:\ 13D02, 20M14, 52B05; Secondary:\ 05E40, 13F20, 13F65}

\keywords{numerical semigroup, toric ideal, free resolution}

\date{18 April 2024}

\begin{abstract}
Numerical semigroups with multiplicity \(m\) are parameterized by integer points in a polyhedral cone~\(C_m\), according to Kunz.  For the toric ideal of any such semigroup, the main result here constructs a free resolution whose overall structure is identical for all semigroups parametrized by the relative interior of a fixed face of~\(C_m\).  The matrix entries of this resolution are monomials whose exponents are parametrized by the coordinates of the corresponding point in~\(C_m\), and minimality of the resolution is achieved when the semigroup is maximal embedding dimension, which is the case parametrized by the interior of~\(C_m\) itself.  
\end{abstract}

\thanks{BB was partially supported by the National Science Foundation
through award DMS-1953785.}

\maketitle


\section{Introduction}\label{sec:intro}

Given a numerical semigroup \(S\), the corresponding semigroup algebra
has a defining toric ideal \(I_S\).  While the study of algebraic
invariants of this numerical semigroup ideal~\(I_S\) falls within 
the broader study of toric ideals,
the family of numerical semigroup ideals forms a rich and
interesting area of study that often affords more refined general
results than are known or possible for the general toric setting.  
The aim of this paper is uniformly constructed explicit free resolutions 
that are minimal for numerical semigroups with maximal embedding dimension, 
are parametrized by Ap\'ery data, and therefore specialize to minimal 
free resolutions for numerical semigroups with arbitrary embedding dimension.  

For general toric ideals, there is a substantial literature on their
resolutions.  In 1998, Peeva and Sturmfels described
minimal free resolutions for generic lattice
ideals~\cite{genericlattice}.  More recently, Tchernev gave an
explicit recursive algorithm for canonical minimal resolutions of
toric rings~\cite{minresrecursive}.  Further, Li, Miller, and Ordog
construct a canonical minimal free resolution of an arbitrary
positively graded lattice ideal with a closed-form combinatorial
description of the differential in characteristic \(0\) and all but
finitely many positive characteristics~\cite{minreslattice}. However,
these constructions are all quite general, and one would hope that in
the special case of numerical semigroups, explicit resolutions of
\(I_S\) more directly tied to the combinatorics of~$S$ are possible.

Free resolutions of $I_S$ for a numerical semigroup $S$ are known in
some special cases.  The surveys~\cite{nssyzygysurvey,nsbettisurvey}
include most results concerning special families.  We outline a few
here.  If $S$ is maximal embedding dimension (MED) and $I_S$ is
determinantal (that is, generated by the minors of a matrix), then
$I_S$ is resolved by the Eagon--Northcott
complex~\cite{gotoencomplex}.  This accounts for some, but not all,
MED numerical semigroups; a~characterization of determinantal MED
numerical semigroups is given in~\cite{determinentalmed}.  If $S$ is
generated by an arithmetic sequence, then $I_S$ is minimally resolved
by a variant of the Eagon--Northcott complex~\cite{nsfreeresarith}.
If $S$ is obtained as a gluing of two numerical semigroups $T$ and
$T'$, then a minimal free resolution of $I_S$ can be obtained from the
minimal free resolutions of $I_T$ and $I_{T'}$ via a mapping cone
construction~\cite{nsfreeresgluing}; this includes the case where
$I_S$ is complete intersection.  If $S$ has at most 3 generators, then
a minimal free resolution is known.  
Numerous families of 4-generated numerical semigroups have also been investigated;
see the survey~\cite{nsbettisurvey} for more detail.

The present work is motivated by recent papers that examine a family of convex rational polyhedra $C_m$ called \emph{Kunz cones}, one for each integer $m \ge 2$, for which each numerical semigroup $S$ with multiplicity $m$---that is, with $m = \min(S \setminus \{0\})$---corresponds to an integer point of $C_m$.  These were first introduced in~\cite{kunz}, and most subsequent papers on the  topic have employed lattice point techniques to examine enumerative questions~\cite{alhajjarkunz,kaplancounting,kunzcoords}.  However, seemingly overlooked for decades was another result in~\cite{kunz} that proved two numerical semigroups $S$ and $T$ correspond to points (relative) interior to the same face of~$C_m$ if and only if certain artinian quotients of $I_S$ and $I_T$ coincide.  Indeed, a corollary of his result, namely that MED numerical semigroups are precisely those lying in the interior of $C_m$, was the only reference to the faces of $C_m$ in the literature until a series of recent papers~\cite{wilfmultiplicity, kunzfaces1} unknowingly rederived a combinatorial version of Kunz's result:\ $S$ and $T$ lie in the same face of $C_m$ if and only if certain subsets of their divisibility posets coincide.  

Kunz also observed in~\cite{kunz} that, as a consequence of his result, $\beta_d(I_S) = \beta_d(I_T)$ for every $d$, though his approach did not allow for explicit construction of syzygies for $I_S$ and $I_T$.  The combinatorial viewpoint of~\cite{wilfmultiplicity,kunzfaces1} was recently used in~\cite{kunzfaces3} to make Kunz's enumerative result algebraic in the case $d = 0$:\ if $S$ and $T$ lie interior to the same face of~$C_m$, then minimal binomial generating sets are explicitly constructed for $I_S$ and $I_T$ that coincide in all terms except the exponents of a single variable.  

Our main result is to similarly make Kunz's result algebraic for all positive $d$:\ 
we~construct explicit free resolutions of all numerical semigroup ideals.  These resolutions are minimal when the semigroup has maximal embedding dimension (equivalently, if $S$ lies in the interior of $C_m$) and are minimalized uniformly for all numerical semigroups lying interior to the same face of the Kunz cone $C_m$.  
More precisely, our contributions are as follows.

\begin{enumerate}[1.]
\item%
We introduce the Ap\'ery toric ideal $J_S$ of $S$, an analogue of
$I_S$ that lies in a ring with $m$ variables instead of a ring with
one variable per minimal generator of $S$.  A~generating set for $J_S$
can be obtained by concatenating any generating set for $I_S$ and a
regular sequence with one element for each additional variable.

\item%
For any positive integer \(m\geq 2\), we construct a free resolution
of \(J_S\) called the \emph{Ap\'ery resolution}.  The rank of the
\(d\)-th free module depends only on \(m\) and \(d\), and the
positions of the nonzero entries of the matrices representing the
boundary maps depend only on~\(m\).  An example of this for \(m=4\) is
given in Figure~\ref{fig:m4res}, where the values \(b_{i,j}\) depend
on \(S\).

\item%
When \(S\) corresponds to a point interior to~\(C_m\), i.e., when
\(S\) is MED and thus $J_S = I_S$, the Ap\'ery resolution is a minimal
free resolution of \(I_S\).

\item\label{i:step4}%
For any numerical semigroups $S$ and~$T$ corresponding to points
interior to the same face~\(F\) of~$C_m$, we prove there exists a
uniform method for modifying the Ap\'ery resolutions of \(J_S\) and
\(J_T\) to minimal free resolutions in such a way that the resulting
ranks of the free modules and the positions of the nonzero entries of
the matrices representing the boundary maps depend only on \(m\),
\(F\), and \(d\).

\end{enumerate}

\begin{figure}[t]
$
\begin{array}{c@{\:}c@{\:}c@{\:}c@{\:}c}
&
\begin{blockarray}{rcccccc}
	\\ \\
	&
	\s \textbf{1,1} &
	\s \textbf{2,2} &
	\s \textbf{3,3} &
	\s \textbf{2,1} &
	\s \textbf{3,1} &
	\s \textbf{3,2} \\
	\begin{block}{r@{\,\,}[*{6}{@{\,}l}]}
		\s \nothing & 
		x_1^2  - x_2 y^{b_{11}} &
		x_2^2  - y^{b_{22}} &
		x_3^2  - x_2 y^{b_{33}} &
		x_1x_2 - x_3 y^{b_{12}} &
		x_1x_3 - y^{b_{13}} &
		x_2x_3 - x_1 y^{b_{23}} \\
	\end{block}
\end{blockarray}
&
\\[-1em]
0 \leftarrow R & \filleftmap &
\end{array}
$

$
\begin{array}{c@{\:}c@{\:}c@{\:}c@{\:}c@{}c}
&
\begin{blockarray}{@{}r*{8}{@{}c}}
  \\ \\
  &
  \s \textbf{1,12} &
  \s \textbf{1,13} &
  \s \textbf{2,12} &
  \s \textbf{2,23} &
  \s \textbf{3,13} &
  \s \textbf{3,23} &
  \s \textbf{2,13} &
  \s \textbf{3,12} \\
  \begin{block}{@{}l@{\,\,\,}[*{8}{@{}l}]}
    \s \textbf{1,1}
       & \rlm x_2        & \rlm x_3        & \phm            & \phm
       & \phm            & \phm            &\,\phm y^{b_{23}}& \phm y^{b_{23}}
  \\
    \s \textbf{2,2}
       & \rlm y^{b_{11}} & \phm            & \,\phm x_1      & \rlm x_3
       & \phm            & \phm y^{b_{33}} & \phm            & \phm
  \\
    \s \textbf{3,3}
       & \phm            & \phm            & \phm            & \phm
       & \phm x_1        & \phm x_2        &\,\rlm y^{b_{12}}& \phm
  \\
    \s \textbf{2,1}
       & \phm x_1        & \phm            & \,\rlm x_2      & \phm y^{b_{23}}
       & \phm y^{b_{33}} & \phm            &\,\rlm x_3       & \phm
  \\
    \s \textbf{3,1}
       & \phm y^{b_{12}} & \phm x_1        & \phm            & \phm
       & \rlm x_3        & \rlm y^{b_{23}} & \phm            & \rlm x_2
  \\
    \s \textbf{3,2}
       & \phm            & \rlm y^{b_{11}} &\,\rlm y^{b_{12}}& \phm x_2
       & \phm            & \rlm x_3        &\,\phm x_1       & \phm x_1
  \\
  \end{block}
\end{blockarray}
&&
\begin{blockarray}{@{}r*{3}{@{}c}}
  &
  \s \textbf{1,[3]} &
  \s \textbf{2,[3]} &
  \s \textbf{3,[3]} \\
  \begin{block}{@{}l@{\,\,}[*{3}{@{\,}l}]}
    \s \textbf{1,12} & \phm x_3        & \rlm y^{b_{23}} &                \\
    \s \textbf{1,13} & \rlm x_2        &                 & \phm y^{b_{23}}\\
    \s \textbf{2,12} &                 & \phm x_3        & \rlm y^{b_{33}}\\
    \s \textbf{2,23} & \rlm y^{b_{11}} & \phm x_1        &                \\
    \s \textbf{3,13} & \phm y^{b_{12}} &                 & \rlm x_2       \\
    \s \textbf{3,23} &                 & \rlm y^{b_{12}} & \phm x_1       \\
    \s \textbf{2,13} & \phm x_1        & \rlm x_2        &                \\
    \s \textbf{3,12} & \rlm x_1        & \phm            & \phm x_3       \\
  \end{block}
\end{blockarray}
&
\\[-1em]
R^6 & \filleftmap & \!R^8 & \filleftmap & \!R^3 & \leftarrow 0
\end{array}
$
\caption{The Ap\'ery resolution for $m = 4$.  The exponents $b_{i,j}$
and $c_{i,j}$ are constants depending on the particular numerical
semigroup $S$.}
\label{fig:m4res}
\end{figure}

The term ``specialization'' in the titles of the paper and
Section~\ref{sec:specialization} refers to passage from the interior
of the Kunz cone to a face, which entails some facet inequalities
becoming equalities.  Consequently, some exponents on $y$ variables (as
in Figure~\ref{fig:m4res}) pass from positive to~$0$, which results in
the specialization that sets $y = 1$.  Further substitutions among the
$x$ variables---extraneous ones are set equal to monomials in the
others---combine in Step~\ref{i:step4} with row and column operations
to produce minimal free resolutions from the original Ap\'ery
resolution.

The remainder of this paper is structured as follows.
Section~\ref{sec:background} reviews basic properties of numerical
semigroups and Kunz cones and defines the modules and maps used in
the Ap\'ery resolution.  Section~\ref{sec:resolution} proves that the
Ap\'ery resolution is indeed a resolution and establishes the
minimality of this resolution when \(S\) is~MED.
Section~\ref{sec:specialization} describes how to modify the Ap\'ery
resolution in a uniform way for all numerical semigroups in the
interior of a fixed face of \(C_m\) to obtain a minimal resolution.
Further research directions are outlined in
Section~\ref{sec:openquesitons}.

\section{Kunz polyhedra and Ap\'ery resolutions}\label{sec:background}

\subsection{Semigroups and toric ideals}\label{b:toric-ideals}

A \emph{numerical semigroup} is a subsemigroup of $(\ZZ_{\ge 0}, +)$
that contains \(0\) and has finite complement.  Throughout this work,
fix a numerical semigroup $S \subset \ZZ_{\ge 0}$ with
\emph{multiplicity}
$$
  \mathsf m(S) = \min(S \setminus \{0\}) = m
$$
and write
\begin{align*}
\Ap(S)
  &= \{n \in S : n - m \notin S\} \\
  &= \{0, a_1, \ldots, a_{m-1}\}
\end{align*}
for the \emph{Ap\'ery set} consisting of the minimal element of $S$
from each equivalence class modulo~$m$, where each \(a_i\) satisfies
$a_i \equiv i \bmod m$.  For convenience, define $a_0 = m$; this
convention plays an important role in our later formulas.  In
particular,
$$
  S = \<m, a_1, \ldots, a_{m-1}\>=\<a_0, a_1, \ldots, a_{m-1}\>,
$$
though this generating set need not be the unique minimal generating
set $\mathcal A(S)$ of $S$, such as when $a_i + a_j = a_{i+j}$ for
some $i, j$, where indices are summed modulo~\(m\).  The semigroup $S$
has \emph{maximum embedding dimension (MED)} if $\mathcal A(S) =
\{a_0, \ldots, a_{m-1}\}$.

\begin{example}\label{ex:semigroup}
The semigroup \(S=\<4,9,11,14\>\) has multiplicity \(\mathsf m(S) =
4\) and Ap\'ery set \(\Ap(S)=\{0,9,14,11\}\).  The semigroup
\(T=\<4,13,23\>\) has multiplicity \(\mathsf m(T) = 4\) and
\(\Ap(T)=\{0,13,26,23\}\).  Note that \(a_1 + a_1 = a_2\) in~\(T\),
and thus the Ap\'ery set is not a minimal generating set.
\end{example}

Let $R = \kk[x_0, x_1, \ldots, x_{m-1}]$ with the natural grading
by~$\ZZ$ via $\deg(x_i) = a_i$ and set~$y = x_0$.  The \emph{Ap\'ery
toric ideal} of~\(S\) is the kernel \(J_S = \ker(\varphi)\) of the
homomorphism
\begin{align*}
  \varphi:R & \longrightarrow \kk[t] \\
        x_i & \longmapsto t^{a_i}
\end{align*}
and the \emph{defining toric ideal} of \(S\) is
$$
  I_S = J_S \cap \kk[x_i : a_i \in \mathcal A(S)].
$$
For every \(1\leq i,j\leq m-1\) define
\begin{equation}\label{eq:cij}
  c_{i,j} = \tfrac{1}{m}(a_i + a_j - a_{i+j}) \ge 0
\end{equation}
and
$$
  b_{i,j}
  =
  \begin{cases}
    c_{i,j}     & \text{if } i + j \ne m \\
    c_{i,j} + 1 & \text{if } i + j  =  m.
\end{cases}
$$
In particular, $c_{i,j} = 0$ if and only if $a_i + a_j = a_{i+j}$;
this is impossible if $i + j = m$, since $m$ is the multiplicity, so
$b_{i,j} = 0$ if and only if $c_{i,j} = 0$.  It is known that
\begin{equation}\label{eq:medbinomials}
  J_S = \<x_ix_j - y^{c_{i,j}}x_{i+j} : 1 \le i \le j \le m-1\>,
\end{equation}
though it also follows from Lemma~\ref{l:idealgens} here.

\begin{example}\label{ex:semigrouptoric}
The semigroup \(S=\<4,9,11,14\>\) has \((a_1,a_2,a_3)=(9,14,11)\) and
$$
  J_S = I_S = \<x_1^2-yx_2,
  x_1x_2-y^3x_3,x_1x_3-y^4y,x_2^2-y^6y,x_2x_3-x_1y^4,x_3^2-x_2y^2\>.
$$
The terms here are written in a way that emphasizes the convention
\(a_0 = m\) and \(x_0 = y\), such as to produce the binomial \(x_1x_3
- y^4x_0 = x_1x_3 - y^4y = x_1x_3 - y^5\).

The semigroup \(T=\<4,13,23\>\) has \((a_1,a_2,a_3)=(13,26,23)\) and
Ap\'ery ideal
\begin{align*}
  J_T &= \<x_1^2-x_2, x_1x_2-y^4x_3, x_1x_3-y^9, x_2^2-y^{13},
         x_2x_3-x_1y^9, x_3^2-x_2y^5\>
\\
      &= \<x_1^2-x_2, x_1^3-y^4x_3, x_1x_3-y^9, x_3^2-x_1^2y^5\>
\end{align*}
and defining toric ideal
$$
  I_T = \<x_1^3-y^4x_3, x_1x_3-y^9, x_3^2-x_1^2y^5\>
      = J_T \cap \kk[y,x_1,x_3].
$$
\end{example}

\subsection{Kunz cone}\label{b:Kunz-cone}

This subsection describes the Kunz cone and its relationship to the
values \(b_{i,j}\).  Letting \(\Ap(S)=\{0,a_1,\ldots,a_{m-1}\}\) with
$a_i \equiv i \bmod m$ for each $i$ as in
Section~\ref{b:toric-ideals}, the \emph{Ap\'ery coordinate vector} of
$S$ with respect to \(m\) is the tuple $(a_1, \ldots, a_{m-1})$.
The following set of linear inequalities exactly characterizes the set
of Ap\'ery coordinate vectors for numerical semigroups of multiplicity
\(m\) \cite{kunzfaces1,kunz}.

\begin{defn}\label{d:kunzcone}
For each \(m\geq 2\), the \emph{Kunz cone} \(C_m \subseteq \RR_{\ge
0}^{m-1}\) has facet inequalities
$$
  z_i + z_j \geq z_{i+j}
  \quad \text{for} \quad
  1 \le i \le j \le m-1
  \quad \text{with} \quad 
  i+j \ne m,
$$
where addition of subscripts is modulo \(m\).  
\end{defn}

\begin{lemma}\label{l:bij=0}
If \(S\) is a numerical semigroup of multiplicity \(m\), then
\(b_{i,j} = 0\) if and only if \(a_i+a_j = a_{i+j}\).
Hence the Ap\'ery coordinate vector of~\(S\) lies on the boundary
of~\(C_m\) if and only if \(b_{i,j} = 0\) for some~\(i,j\).
\end{lemma}
\begin{proof}
Follows from the definitions, using~(\ref{eq:cij}) for the claim
about~$b_{i,j}$.
\end{proof}

The lemma has the following consequence \cite{kunz}.

\begin{prop}\label{prop:kunzinteriormed}
A vector $z = (z_1, \ldots, z_{m-1}) \in \ZZ_{\ge 1}^{m-1}$ with $z_i
\equiv i \bmod m$ for all~$i$ lies in $C_m$ if and only if $z$ is the
Ap\'ery coordinate vector of a numerical semigroup~$S$.  Moreover, $z$
is in the interior of $C_m$ if and only if \(S\) has maximal embedding
dimension.
\end{prop}

\begin{example}\label{ex:kunzm4}
The cone $C_4 \subseteq \RR_{\ge 0}^3$ is defined by the inequalities
$$
z_2 + z_3 \ge z_1, 
\qquad
z_1 + z_2 \ge z_3, 
\qquad
2z_1 \ge z_2, 
\qquad \text{and} \qquad
2z_3 \ge z_2,
$$
and has extremal rays generated by $(1,0,1)$, $(1,2,3)$, $(1,2,1)$,
and $(3,2,1)$.  All positive-dimensional faces of $C_4$ contain
numerical semigroups (in the sense of
Proposition~\ref{prop:kunzinteriormed}) except the rays through
$(1,0,1)$ and $(1,2,1)$.  Numerical semigroups on the rays through
$(1,2,3)$ and $(3,2,1)$ have embedding dimension 2, and numerical
semigroups in the relative interior of the facets $z_1 + z_2 = z_3$
and $z_2 + z_3 = z_1$ are complete intersections~\cite{kunzfaces2}.
In particular, a minimal free resolution for the defining toric ideal
of any semigroup in these $4$ faces is known.

The numerical semigroup $S$ from Example~\ref{ex:semigrouptoric}
corresponds to the point $(9,14,11)$ in the relative interior of
$C_4$, while $T$ corresponds to the point $(13,26,23)$ in the relative
interior of the facet $2z_1 = z_2$.  A minimal free resolution for
$J_S$ is obtained by substituting the appropriate values for $b_{i,j}$
in the free resolution in Figure~\ref{fig:m4res}, while a minimal free
resolution for $J_T$ is obtained via analogous substitution into
Figure~\ref{fig:m4spec} (at the end of
Section~\ref{sec:specialization}).  This leaves the facet $2z_3 =
z_2$, and, courtesy of the action of $\ZZ_4^*$ on $C_4$, free
resolutions for semigroups in this face can be obtained from the ones
exhibited in Figure~\ref{fig:m4spec} by interchanging 1's and 3's in
every subscript.
\end{example}

We record here the result of Kunz that seems to be overlooked in the literature concerning when numerical semigroups reside in the interior of a given face of $C_m$.  

\begin{thm}[{\cite[Propositions~2.3 and~2.6]{kunz}}]\label{t:kunzbetti}
Two numerical semigroups $S$ and $T$ with multiplicity $m$ lie in the interior of the same face of $C_m$ if and only if 
$$R/(J_S + \<y\>) \cong R/(J_T + \<y\>).$$
Moreover, in this case, $\beta_d(I_S) = \beta_d(I_T)$ for every $d$.  
\end{thm}

\subsection{Modules and maps for the Ap\'ery resolution}\label{b:modules}

For any numerical semigroup~\(S\) of multiplicity~\(m\), this
subsection defines the free modules and linear maps between them that
form the \emph{Ap\'ery resolution}
$$
  \mathcal F_\bullet: 0 \longleftarrow R \longleftarrow F_1
  \longleftarrow F_2 \longleftarrow \cdots
$$
of~\(J_S\).  Theorem~\ref{t:medresolution} shows that it is a
resolution.  Of particular note is that the ranks of its modules and
the locations of the nonzero coefficients in the matrices representing
its linear maps depend only on~\(m\), not on the actual values of
\(\Ap(S)\).

Theorem~\ref{t:medresolution} and Corollary~\ref{c:minimal-for-med}
show that this resolution is minimal if and only if \(S\) has maximal
embedding dimension, i.e., corresponds to a point interior to~\(C_m\).
Theorem~\ref{t:specializationindependent} shows that when \(S\) lies
on the boundary of~\(C_m\), a minimal resolution of~$J_S$ can be
obtained from the Ap\'ery resolution in a manner that is uniform for
all semigroups in the interior of a fixed face of~\(C_m\),
parametrized by the $b_{i,j}$.

\subsubsection{Modules}\label{ss:modules}

For \(d=0,1,\ldots,m-1\), define $F_d$ to be the free module over
\(R\) with formal basis elements
$$
  \left\{e_{i,A} : i \in [m-1], \, A \subset [m-1], |A| = d, i \ge \min(A)\right\},
$$
where $\deg(e_{i,A}) = a_i + \sum_{j \in A} a_j$.  Since every pair
\((i,A)\) such that \(|A|=d\) and \(i<\min(A)\) corresponds to a
\((d+1)\)-element subset \(\{i\}\cup A\) of \([m-1]\), it is immediate
that
\begin{align*}
\rank F_d 
&= (m-1)\binom{m-1}{d} - \binom{m-1}{d+1}
= d\binom{m}{d+1} \, .
\end{align*}

\begin{example}\label{e:module}
For \(m=3\), \(F_0 = Re_\nothing\), where \(Re_\nothing =
\{re_\nothing : r\in R\}\).  Similarly,
\begin{align*}
  F_1 &= Re_{1,\{1\}}+Re_{2,\{2\}}+Re_{2,\{1\}}
\\
      &= \{\alpha e_{1,\{1\}}+\beta e_{2,\{2\}}+\gamma e_{2,\{1\}}:
         \alpha,\beta,\gamma \in R\} 
\end{align*}
with \(\deg(e_{1,\{1\}})= a_1+a_1\), \(\deg(e_{2,\{2\}})=a_2+a_2 \),
and \(\deg(e_{2,\{1\}})=a_2+a_1 \).  Note that \(\rank F_1 =
1\cdot\binom{3}{1+1}\).  Finally,
$$
  F_2 = Re_{1,12}+Re_{2,12}
$$
with $\deg(e_{1,12}) = a_1+a_1+a_2$ and $\deg(e_{2,12}) =
a_2+a_1+a_2$.
\end{example}

\subsubsection{Maps}\label{ss:maps}

A few notational conventions help to define the boundary maps between
the \(F_d\).  For $A \subseteq [m-1]$, set
$$
  \sign(j,A)=(-1)^t
  \qquad \text{for} \qquad
  j\in A=\{\ell_0<\ell_1<\cdots <\ell_t=j <\cdots <\ell_r\}.  
$$
For convenience, set $e_{0,A} = 0$, and for $i \in [m-1]$ with $i <
\min(A)$, define
\begin{equation}\label{eq:quotientsub}
  e_{i,A} = \sum_{j \in A} \sign(j,A) e_{j,A \cup i \setminus j}\, .
\end{equation}
As a consequence of this definition of \(e_{i,A}\), for each $B
\subseteq [m-1]$,
\begin{equation}\label{eq:quotientrel}
  \sum_{i \in B} \sign(i,B) e_{i,B \setminus i} = 0.
\end{equation}

With these conventions in hand, and considering $i+j$ modulo $m$ in
subscripts as usual, define the map $\partial_d: F_d \to F_{d-1}$ by
\begin{equation}\label{eq:definepartial}
  e_{i,A} \mapsto \sum_{j \in A} \sign(j, A) (x_j e_{i,A \setminus j}
  - y^{b_{i,j}} e_{i+j,A \setminus j})
\end{equation}
with the exception of $d = 1$, in which case
$$
\partial_1(e_{i,j}) = x_ix_j - y^{c_{i,j}} x_{i+j}\, .
$$

\begin{example}\label{ex:m3res}
Figure~\ref{fig:m3res} shows the modules and maps for the case
\(m=3\).  Note that by definition the bases for the modules are
indexed by \((i,A)\) pairs, and these are used to label the rows and
columns of the matrices representing the maps.  Consider the term
\(\partial_2(e_{1,12})\), which by definition is
\begin{align*}
\partial_2(e_{1,12})
  & = \left(x_1e_{1,2}-y^{b_{1,1}}e_{2,2}\right)
      - \left(x_2e_{1,1}-y^{b_{1,2}}e_{0,1}\right)  \\
  & = x_1e_{2,1} - y^{b_{1,1}} e_{2,2} - x_2e_{1,1}
\end{align*}
where the relation \(e_{1,2} - e_{2,1} = 0\) is used.  This
illustrates how the relation~\eqref{eq:quotientrel} ensures that
\(\partial_d\) is well defined.
\end{example}

\begin{example}\label{ex:m4res}
Figure~\ref{fig:m4res} shows the modules and maps for the case
\(m=4\).  The ranks of the modules and the general structure of the
maps are independent of the values of~\(\Ap(S)\) except that the
exponents on the \(y\)-variables in the matrices. The Ap\'ery
resolution of $I_S$ for $S$ introduced in Example~\ref{ex:semigroup}
is given in the upper portion of Figure~\ref{fig:determinantalEN}
(toward the end of Section~\ref{sec:resolution}).
\end{example}
  
\begin{figure}[t]
\begin{center}
$
\begin{array}{c@{\:}c@{\:}c@{\:}c@{\:}c}
&
\begin{blockarray}{rccc}
  \\ \\
  &
  \s \textbf{1,1} &
  \s \textbf{2,2} &
  \s \textbf{2,1} \\
  \begin{block}{r@{\,\,\,}[lll]}
    \s \nothing & 
    x_1^2 - x_2 y^{b_{11}} &
    x_2^2 - x_1 y^{b_{22}} &
    x_1x_2 - y^{b_{12}} \\
  \end{block}
\end{blockarray}
&&
\begin{blockarray}{rc@{\,\,}c}
  &
  \s \textbf{1,12} &
  \s \textbf{2,12} \\
  \begin{block}{r@{\,\,\,}[l@{\,\,}l]}
    \s \textbf{1,1} & \rlm x_2       & \phm y^{b_{22}} \\
    \s \textbf{2,2} & \rlm y^{b_{11}} & \phm x_1 \\
    \s \textbf{2,1} & \phm x_1       & \rlm x_2 \\
  \end{block}
\end{blockarray}
&
\\[-1em]
0 \leftarrow R & \filleftmap & R^3 & \filleftmap & R^2 \leftarrow 0
\end{array}
$
\end{center}
\caption{The Ap\'ery resolution for $m = 3$.}
\label{fig:m3res}
\end{figure}

\section{Maximal embedding dimension numerical semigroups}\label{sec:resolution}

This section proves the Ap\'ery resolution to be indeed a resolution
of~$J_S$.  The core of the proof is Schreyer's theorem, which
identifies a Gr\"obner basis (under a carefully chosen term order) for
the syzygy module of a Gr\"obner basis (see~\cite{clo2} for a thorough
overview of Schreyer's theorem).  Lemma~\ref{l:quotient} verifies
important subtleties about the boundary maps $\partial_d$: they are
consistent with the definition of $e_{i,A}$ in~\eqref{eq:quotientsub}
and~\eqref{eq:quotientrel} when $i < \min A$, and
substituting~\eqref{eq:quotientsub} into the definition of
$\partial_d$ still yields matrix entries that are monomials.

\begin{remark}\label{r:schreyerresolution}
The authors of~\cite{eisenbudschreyerfirst} produce a resolution that is isomorphic to the one defined in Section~\ref{sec:background}.  In fact, Frank-Olaf Schreyer informed us (in personal communication) that resolution of the ideals in~\eqref{eq:medbinomials} was the initial motivation for both~\cite{eisenbudschreyerfirst} and what is now known as Schreyer's theorem (see \cite[Chapter~5, (3.3)]{clo2}).  
However, the particular form taken by our explicit matrices clarifies the sense in which our resolution is compatible with specialization in the sense of Theorem~\ref{t:specializationindependent}.  This point has some subtlety:\ obtaining a resolution isomorphic to ours need not be sufficient for the purpose of specialization (see, for instance, the resolutions and discussion in Remark~\ref{r:noteagonnorthcott}).  As~such, we include here a full proof of Theorem~\ref{t:medresolution}.  
\end{remark}

\begin{lemma}\label{l:idealgens}
The generating set~\eqref{eq:medbinomials} is a Gr\"obner
basis for $J_S$ under any term order~$\preceq$ 
on~$R$ for which $\xx^\aa y^r \succ \xx^\bb y^s$ whenever 
$a_1 + \cdots + a_{m-1} > b_1 + \cdots + b_{m-1}$, 
where~$\xx^\aa$ and~$\xx^\bb$ are monomials in $x_1,\dots,x_{m-1}$.
\end{lemma}
\begin{proof}
Since $J_S$ is generated by binomials, it suffices to consider
binomials when computing initial ideals.  The key observation is that
in any graded degree, exactly one monomial in $R$ has the form $x_i
y^a$ with $a \in \ZZ_{\ge 0}$, since the graded degrees of the
variables $x_i$ are distinct modulo~$m$.  Hence the larger term
under~$\preceq$ in any nonzero binomial from~$J_S$ is divisible by
$x_ix_j = \In_\preceq(x_ix_j - y^{c_{i,j}} x_{i+j})$ for some $i, j
\in [m-1]$.  As~such,
$$
  \In_\preceq(J_S) = \<x_ix_j : 1 \le i, j \le m-1\>,
$$
and thus the generating set in~\eqref{eq:medbinomials} is a Gr\"obner
basis for $J_S$.
\end{proof}

The next aim is to establish that applying $\partial$ to $e_{i,A}$
when $i<\min(A)$ using the expression given
in~\eqref{eq:definepartial} is consistent with~\eqref{eq:quotientrel};
this is needed when considering the result of applying $\partial$
repeatedly.  Further, careful analysis of the use
of~\eqref{eq:quotientrel} is required when $i+j<\min(A\setminus j)$
in~\eqref{eq:definepartial}.  These issues are addressed in the
following lemma.

\begin{lemma}\label{l:quotient}
The maps $\partial_d$ respect~\eqref{eq:quotientrel}, and for $d > 1$,
the entries of each $\partial_d$ are monomials.  Furthermore,
$\mathcal F_\bullet$ is a complex.
\end{lemma}
\begin{proof}
If $d = 1$, then~\eqref{eq:quotientsub} yields $e_{i,j} = e_{j,i}$,
and the first claim is immediate.  If $d > 1$, then for each $B
\subseteq [m-1]$ with $|B| = d+1$,
$$
  \sum_{i \in B} \sign(i,B) \partial e_{i,B \setminus i}
  =
  \sum_{i \in B} \sign(i,B)
  \sum_{j \in B \setminus i} \sign(j, B \setminus i)
    (x_j e_{i,B \setminus ij} - y^{b_{i,j}} e_{i+j,B \setminus ij}),
$$
wherein the coefficient of $y^{b_{i,j}} e_{i+j,B \setminus ij}$ for
distinct $i, j \in B$ equals
$$
  \sign(i,B) \sign(j,B \setminus i) + \sign(j,B) \sign(i,B \setminus j) = 0
$$
and the remaining terms yield
\begin{align*}
\sum_{i \in B} \sign(i,B) \partial e_{i,B \setminus i}
  &= \sum_{i \in B} \sign(i,B) \sum_{j \in B \setminus i} \sign(j, B
     \setminus i) x_j e_{i,B \setminus ij}
\\
  &= \sum_{j \in B} x_j \sum_{i \in B \setminus j} \sign(i,B)\sign(j,B
     \setminus i) e_{i,B \setminus ij}
\\
  &= -\sum_{j \in B} \sign(j,B) x_j \sum_{i \in B \setminus j}
     \sign(i,B \setminus j) e_{i,B \setminus ij}
\\
  &= -\sum_{j \in B} \sign(j,B) x_j \cdot 0
\\
  &= 0.
\end{align*}

Now proceed to the second claim that each $\partial_d$ is a matrix
whose entries are monomials.  Call $e_{i,A}$ \emph{squarefree} if $i
\notin A$.  Note that every term in~\eqref{eq:quotientrel} is
squarefree, and no two equalities of the form~\eqref{eq:quotientrel}
share any terms.  As such, it suffices to ensure that no two
squarefree terms in $\partial e_{i,A}$ lie in the same equality
in~\eqref{eq:quotientrel}.  To this end, fix $j, k \in A$.  If $e_{i,
A \setminus j}$ and $e_{i, A \setminus k}$ lie in the same equality
in~\eqref{eq:quotientrel}, then $(A \setminus j) \cup i = (A \setminus
k) \cup i$ and thus $j = k$.  If $e_{i+j, A \setminus j}$ and $e_{i+k,
A \setminus k}$ lie in the same equality in~\eqref{eq:quotientrel},
then $i + j \notin A$ but $i + j \in (A \cup \{i + k\}) \setminus k$,
so necessarily $i + j = i + k$ and thus $j = k$.  Lastly, if $e_{i, A
\setminus j}$ and $e_{i+k, A \setminus k}$ lie in the same equality
in~\eqref{eq:quotientrel}, then $i+k \notin A$ but $i+k \in (A
\setminus j) \cup i$, so $i + k = i$, which is impossible.

It remains to prove $\mathcal F_\bullet$ is a complex.  First, suppose
$A = \{j,k\}$ with $j < k$ and $i \ge j$.  If $i + j$, $i + k$, and $i
+ j + k$ are all nonzero, then
\begin{align*}
\partial^2 e_{i,A}
  &= x_j \partial e_{i,k} - x_k \partial e_{i,j} - y^{b_{i,j}} \partial
     e_{i+j,k} + y^{b_{i,k}} \partial e_{i+k,j}
\\
  &= x_j (x_ix_k - y^{c_{i,k}}x_{i+k}) - x_k (x_ix_j - y^{c_{i,j}}x_{i+j})
  \\&\phantom{={}} - y^{b_{i,j}} (x_{i+j}x_k - y^{c_{i+j,k}}x_{i+j+k}) +
     y^{b_{i,k}} (x_{i+k}x_j - y^{c_{i+k,j}}x_{i+j+k})
\\
  &= x_{i+j}x_k (y^{c_{i,j}} - y^{b_{i,j}}) + x_{i+k}x_j (y^{c_{i,k}} -
     y^{b_{i,k}})
  \\&\phantom{={}} + x_{i+j+k} (y^{b_{i,j}} y^{c_{i+j,k}} - y^{b_{i,k}}
     y^{c_{i+k,j}}) = 0
\end{align*}
by homogeneity of $\partial$.  In the event $i + j = 0$, or $i + k =
0$, or $i + j + k = 0$, replacing $x_0$ with zero as appropriate in
the above algebra yields the desired equality.  For all remaining
cases, $|A| > 2$, and in the expansion of
\begin{align*}
\partial^2 e_{i,A}
&= \sum_{j \in A} \sign(j, A) x_j \partial e_{i,A \setminus j} -
\sum_{j \in A} \sign(j, A) y^{b_{i,j}} \partial e_{i+j,A \setminus j}
\\
&= \sum_{j \in A} \sign(j, A) x_j \bigg( \sum_{k \in A \setminus j}
\sign(k, A \setminus j) (x_k e_{i,A \setminus jk} - y^{b_{i,k}}
e_{i+k,A \setminus jk}) \bigg) \\*
& \phantom{={}} - \sum_{j \in A} \sign(j, A) y^{b_{i,j}} \bigg(
\sum_{k \in A \setminus j} \sign(k, A \setminus j) (x_k e_{i+j,A
\setminus jk} - y^{b_{i+j,k}} e_{i+j+k,A \setminus jk}) \bigg),
\end{align*}
the terms $x_j x_k e_{i,A \setminus jk}$, $x_j y^{b_{i,k}}e_{i+k,A
\setminus jk}$, and $y^{b_{i,j} + b_{i+j,k}} e_{i+j+k, A \setminus
jk}$ each have coefficient
$$
  \sign(j,A)\sign(k,A \setminus j) + \sign(k,A)\sign(j, A \setminus k) = 0
$$
for any distinct $j, k \in A$.  
\end{proof}

\begin{thm}\label{t:medresolution}
The complex $\mathcal F_\bullet$ is a resolution.  
\end{thm}
\begin{proof}
Proceed by induction on $d$ to show that the columns of the matrices
for $\partial_d$ form a Gr\"obner basis for $\ker\partial_{d-1}$.  The
case $d = 1$ is handled by Lemma~\ref{l:idealgens}, so suppose $d =
2$.  Let $\preceq'$ denote the partial order on $F_1$ given by
$x^\beta e_{k,\ell} \preceq' x^\alpha e_{i,j}$ whenever
$$
  \In_\preceq(\partial_1(x^\beta e_{k,b})) \prec
  \In_\preceq(\partial_1(x^\alpha e_{i,a})),
$$
or when equality holds above and $a< b $, or if equality holds above,
$a = b$, and $i < k$.  Note $x^\beta e_{j,\ell} \preceq' x^\alpha
e_{i,k} $ whenever $x^\alpha$ has higher total degree in $x_1, \ldots,
x_{m-1}$ than $x^\beta$, so
$$
\In_{\preceq'}(\partial_2(e_{i,jk})) = x_k e_{i,j}
\qquad \text{where} \qquad
j < k \text{ and } i \ge j.$$
Schreyer's theorem~\cite[Chapter 5, (3.3)]{clo2} implies the elements
\begin{equation}\label{eq:schreyergb}
  s_{i,a; \, k,b} = \frac{L}{x_ix_a} e_{i,a} - \frac{L}{x_kx_b}
  e_{k,b} - \sum_{\ell \ge c \ge 1} f_{\ell,c} e_{\ell,c}
\qquad \text{for} \qquad
i \ge a, \, k \ge b
\end{equation}
form a Gr\"obner basis for $\ker(\partial_1)$ under $\preceq'$, where
$L = \lcm(x_ix_a, x_kx_b)$ and the $f_{\ell,c}$ are coefficients
obtained from polynomial long division when dividing
$S(\partial_1(e_{i,a}), \partial_1(e_{k,b}))$
by~\eqref{eq:medbinomials}.  In particular, we claim
$$
  \In_{\preceq'}(\ker(\partial_1))
  =
  \<x_k e_{i,j} \mid j < k \text{ and } i \ge j\>
$$
is generated by initial terms of the columns of $\partial_2$.  Indeed,
by construction $\In_{\preceq'}(s_{i,a; \, k,b})$ must be one of the
first two terms in~\eqref{eq:schreyergb} and
$\partial_1(\frac{L}{x_ix_a} e_{i,a}) =\partial_1( \frac{L}{x_kx_b}
e_{k,b} )$, so without loss of generality say $e_{k,b} \prec'
e_{i,a}.$ Then either $a<b \leq k$ and $x_k$ or $x_b$ appear as a
coefficient of $e_{i,a}$, or $a=b$, $a\leq i <k$ and $x_k$ appears as
a coefficient of $e_{i,a}$, so $\In_{\preceq'}(s_{i,a; \, k,b})$ is
divisible by the initial term of some column of~$\partial_2$.  This
implies that $\In_{\preceq'}(\image(\partial_2)) =
\In_{\preceq'}(\ker(\partial_1))$, which, together with
$\image(\partial_2) \subseteq \ker(\partial_1)$, implies
$\image(\partial_2) = \ker(\partial_1)$ and the columns of
$\partial_2$ form a Gr\"obner basis under $\preceq'$.

Lastly, suppose $d > 2$, let $\preceq$ denote the term order on
$F_{d-2}$ obtained inductively, and let $\preceq'$ denote the term
order on $F_{d-1}$ so that $x^\beta e_{j,B}\preceq' x^\alpha e_{i,A}$
whenever
$$
  \In_\preceq(x^\beta \partial_{d-1}(e_{j,B}))
  \prec
  \In_\preceq(x^\alpha \partial_{d-1}(e_{i,A})),
$$
or if equality holds above and $A$ precedes $B$ lexicographically, or
if equality holds above, $A = B$, and $i < j$.  One readily obtains
$$
\In_{\preceq'}(\partial_d(e_{i,A})) = x_j e_{i,A \setminus j}
\qquad \text{with} \qquad
j = \max(A)
$$
after checking the following: 
\begin{itemize}
\item%
$x^\alpha e_{k,B} \preceq' x^\beta e_{\ell,C}$ whenever $x^\beta$ has
higher total degree in $x_1, \ldots, x_{m-1}$ than $x^\alpha$;

\item%
$x_k \In_\preceq(\partial_{d-1}(e_{i,A \setminus k})) = x_\ell
\In_\preceq(\partial_{d-1}(e_{i,A \setminus \ell}))$ for all $k, \ell
\in A$; and

\item%
the substitution~\eqref{eq:quotientsub} need only be made if $\min(A)
\le i < \min(A \setminus \min(A))$, in which case $A \setminus j$
lexicographically precedes the second subscript of every summand
in~\eqref{eq:quotientrel}.
\end{itemize}

The equality $S(\partial_{d-1}(e_{i,A}), \partial_{d-1}(e_{j,B})) = 0$
holds due to initial terms having distinct basis vectors unless $i =
j$, $A = C \cup \{\gamma\}$, and $B = C \cup \{\delta\}$ for some
$\delta, \gamma \in [m-1]$ and some nonempty $C \subseteq [m-1]$ with
$\delta, \gamma > \max(C)$.  As such, Schreyer's theorem yields
$$
  \In_{\preceq'}(\ker(\partial_{d-1})) = \<x_\delta e_{i, C} : i,
  \delta \in [m-1], \, C \subseteq [m-1], \, \delta > \max(C)\>,
$$
and since $x_\delta e_{i, C} = \In_{\preceq'}(\partial_d(e_{i,C \cup
\delta}))$ for each $i$, $\delta$, and $C$, the columns
of~$\partial_d$ form a Gr\"obner basis for $\ker(\partial_{d-1})$.
The proof is completed by observing that induction also ensures none
of the intial terms in question involve $e_{i,A}$ with $i < \min(A)$.
\end{proof}

\begin{cor}\label{c:minimal-for-med}
The resolution $\mathcal F_\bullet$ is minimal if and only if $S$ is MED.  
\end{cor}

\begin{proof}
A resolution is minimal if and only if the matrices for $\partial_d$
contain no nonzero constant entries.  The only entries that depend on
$a_1, \ldots, a_{m-1}$ are powers of $y$, and their exponents
$b_{i,j}$ are all strictly positive precisely when $S$ is MED.
\end{proof}

\begin{remark}\label{r:noteagonnorthcott}
MED semigroups whose associated toric ideal is determinantal are
exactly those semigroups where $a_1, a_2, \ldots, a_{m-1}$ form an
arithmetic sequence (not necessarily in that order)
\cite{gotoencomplex, determinentalmed}.  In this case, $I_S$ is
resolved by the Eagon--Northcott complex~\cite{eagonnorthcott}; a
detailed treatment on the Eagon--Northcott resolution can be found
in~\cite[Appendix~A2H]{Eis05}.  The strict requirements on an MED
semigroup to make its associated toric ideal determinantal mean that
such semigroups form only a small proportion of all numerical
semigroups:\ in the Kunz cone, these semigroups lie in the union of a
finite set of affine $2$-planes, whose union cannot be the whole cone.
Although relatively few toric ideals of MED semigroup ideals are
minimally resolved by Eagon--Northcott complexes, the occasional
overlap
does mean that all toric ideals for MED numerical semigroups share
Betti numbers with the Eagon--Northcott resolution of a $2 \times m$
matrix, despite the impossibility of using the Eagon--Northcott
construction to resolve most such toric ideals.

Even in the case where the ideal is determinantal, the Ap\'ery
resolution differs from the Eagon--Northcott resolution. As an
example, consider the numerical semigroup $S = \langle 4, 9, 10, 11
\rangle$, whose defining toric ideal $I_S$ is generated by the $2
\times 2$ minors of
$$
\begin{bmatrix}
x_1 & x_2 & x_3 & y^3 \\
y^2 & x_1 & x_2 & x_3
\end{bmatrix}.
$$
The key difference is the presentation of the generators of
$I_S$. Namely, the generators as provided in \eqref{eq:medbinomials}
are of the form $x_ix_j - x_{i+j}y^{b_{i,j}}$, while those given by
determinants may have the form $x_{i}x_j - x_{i+1}x_{j-1}$.
Figure~\ref{fig:determinantalEN} shows the Ap\'ery resolution and the
Eagon--Northcott resolution of $I_S$, with basis elements in the
Eagon--Northcott resolution ordered to mimic the Ap\'ery resolution.
It is worth noting that in the $m=3$ case, $a_1$ and $a_2$ trivially
form an arithmetic sequence, and in fact the Ap\'ery resolution and
the Eagon--Northcott resolution coincide.
\begin{figure}[t]
  $
  \begin{array}{c@{\:}c@{\:}c@{\:}c@{\:}c}
    &
    \begin{blockarray}{cccccc}
      \\ \\
      \begin{block}{[*{6}{@{\,\,}l}]}
        x_1^2  - x_2 y^2 &
        x_2^2  - y^5 &
        x_3^2  - x_2 y^3 &
        x_1x_2 - x_3 y^2 &
        x_1x_3 - y^5 &
        x_2x_3 - x_1 y^3 \\
      \end{block}
    \end{blockarray}
    &
    \\[-1em]
    0 \leftarrow R & \filleftmap &
  \end{array}
  $

\medskip  

  $
  \begin{array}{c@{\:}c@{\:}c@{\:}c@{\:}c@{}c}
    &
    \begin{blockarray}{@{}*{8}{@{\,\,}c}}
      \\ \\
      \begin{block}{@{}[*{8}{@{\,\,}l}]}
        \rlm x_2        & \rlm x_3        & \phm            & \phm
      & \phm            & \phm            & \phm y^3        & \phm y^3
      \\
        \rlm y^2        & \phm            & \phm x_1        & \rlm x_3
      & \phm            & \phm y^3        & \phm            & \phm
      \\
        \phm            & \phm            & \phm            & \phm
      & \phm x_1        & \phm x_2        & \rlm y^2        & \phm
      \\
        \phm x_1        & \phm            & \rlm x_2        & \phm y^3
      & \phm y^3        & \phm            & \rlm x_3        & \phm
      \\
        \phm y^2        & \phm x_1        & \phm            & \phm
      & \rlm x_3        & \rlm y^3        & \phm            & \rlm x_2
      \\
        \phm            & \rlm y^2        & \rlm y^2        & \phm x_2
      & \phm            & \rlm x_3        & \phm x_1        & \phm x_1
      \\
      \end{block}
    \end{blockarray}
    &&
    \begin{blockarray}{@{}*{3}{@{\,\,}c}}
      \begin{block}{@{}[*{3}{@{\,\,}l}]}
        \phm x_3        & \rlm y^3 &                  \\
        \rlm x_2        &                 & \phm y^3  \\
                        & \phm x_3        & \rlm y^3  \\
        \rlm y^2 & \phm x_1        &                  \\
        \phm y^2 &                 & \rlm x_2         \\
                        & \rlm y^2 & \phm x_1         \\
        \phm x_1        & \rlm x_2        &                  \\
        \rlm x_1        & \phm            & \phm x_3         \\
      \end{block}
    \end{blockarray}
    &
    \\[-1em]
    R^6 & \filleftmap & R^8 & \filleftmap & R^3 & \leftarrow 0
  \end{array}
$

\bigskip
\bigskip
\bigskip

  $
  \begin{array}{c@{\:}c@{\:}c@{\:}c@{\:}c}
    &
    \begin{blockarray}{cccccc}
      \begin{block}{[*{6}{@{\,\,}l}]}
        x_1^2  - x_2 y^2 &
        x_2^2  - x_1x_3 &
        x_3^2  - x_2 y^3 &
        x_1x_2 - x_3 y^2 &
        x_1x_3 - y^5 &
        x_2x_3 - x_1 y^3 \\
      \end{block}
    \end{blockarray}
    &
    \\[-1em]
    0 \leftarrow R & \filleftmap &
  \end{array}
  $

\medskip  
  
  $
  \begin{array}{c@{\:}c@{\:}c@{\:}c@{\:}c@{}c}
    &
    \begin{blockarray}{*{8}{@{\,\,}c}}
      \\ \\
      \begin{block}{[*{8}{@{\,\,}l}]}
      \phm x_2 &\phm x_3 &\phm x_3 &\phm     &\phm     &\phm     &\phm     &\phm y^3
      \\
      \phm y^2 &\phm     &\phm x_1 &\phm x_3 &\phm     &\phm y^3 &\phm     &\phm
      \\
      \phm     &\phm     &\phm     &\phm x_1 &\phm x_1 &\phm x_2 &\phm y^2 &\phm
      \\
      \rlm x_1 &\phm     &\rlm x_2 &\phm     &\phm y^3 &\phm     &\phm x_3 &\phm
      \\
      \phm     &\rlm x_1 &\phm     &\phm     &\rlm x_3 &\phm     &\rlm x_2 &\rlm x_2
      \\
      \phm     &\phm y^2 &\phm     &\rlm x_2 &\phm     &\rlm x_3 &\phm     &\phm x_1
      \\
      \end{block}
    \end{blockarray}
    &&
    \begin{blockarray}{*{3}{@{\,\,}c}}
      \begin{block}{[*{3}{@{\,\,}l}]}
        \rlm x_3        & \rlm y^3 &                  \\
        \phm x_2        & \phm x_3 & \phm             \\
        \phm & \rlm x_3 & \rlm y^3 &                  \\
        \phm y^2        & \phm x_1 &                  \\
        \phm            & \rlm x_1 & \rlm x_2         \\
        \phm & \phm y^2 & \phm x_1 &                  \\
        \rlm x_1        & \rlm x_2 &                  \\
        \phm            & \phm x_2 & \phm x_3         \\
      \end{block}
    \end{blockarray}
    &
    \\[-1em]
    R^6 & \filleftmap & R^8 & \filleftmap & R^3 & \leftarrow 0
  \end{array}
  $
\caption{The Ap\'ery resolution (above) and Eagon--Northcott
resolution (below) for $I_S$ where $S = \langle4,9,10,11\rangle$}
\label{fig:determinantalEN}
\end{figure}
\end{remark}

\begin{remark}\label{r:artinianreduction}
When $S$ is MED, quotienting the Ap\'ery resolution of~$I_S$ by the ideal~$\<y\>$, as Kunz does in Theorem~\ref{t:kunzbetti} with the ring $R/I_S$, yields a minimal resolution of the ideal $\<x_1, \ldots,x_{m-1}\>^2$ over the ambient polynomial ring $\kk[x_1, \ldots,x_{m-1}]$.  This ideal is known to be resolved by the Eagon--Northcott complex on the $2 \times m$ matrix
$$
\begin{bmatrix}
x_1 & x_2 & \cdots & x_{m-2} & x_{m-1} & 0\\
0 & x_ 1 &x_2 & x_3 & \cdots & x_{m-1}
\end{bmatrix}.
$$
Thus, in the MED case, the Eagon--Northcott complex ``sits inside'' the Ap\'ery resolution; indeed, it is the result of an artinian reduction of~$R/I_S$.
\end{remark}

\section{Specialization for arbitrary numerical semigroups}\label{sec:specialization}

The Ap\'ery resolution can be thought of as a family of free
resolutions, one for the Ap\'ery ideal $J_S$ of each numerical
semigroup $S$ with multiplicity $m$, that is parametrized by the
values $b_{i,j}$.  Given a numerical semigroup $S$, a free resolution
of $J_S$ is obtained by simply computing the values $b_{i,j}$ from the
Ap\'ery set of $S$ and substituting them into the Ap\'ery resolution.
By Corollary~\ref{c:minimal-for-med}, restricting to semigroups $S$ in
the interior of~$C_m$, the Ap\'ery resolutions form a parametrized
family of minimal free resolutions.

The main result of this section is
Theorem~\ref{t:specializationindependent}, which implies that for each
face $F$ of~$C_m$, there exists a family of minimal free resolutions,
one for the Ap\'ery ideal $J_S$ of each numerical semigroup $S$
indexed by the interior of~$F$, that is analogously parametrized by
the positive $b_{i,j}$.  Figure~\ref{fig:m4spec} depicts one such
resolution for the $z_2 = 2z_1$ facet of $C_4$.  Our proof of
Theorem~\ref{t:specializationindependent} is nonconstructive: it
carefully argues that there exists a change of basis for the Ap\'ery
resolution, depending only on $F$, that yields the desired minimal
free resolution of $J_S$ as a summand.  Together with
Proposition~\ref{p:freeresextravars}, which gives the algebraic
relationship between minimal resolutions of~$J_S$ and~$I_S$,
the Betti numbers of~$J_S$ and~$I_S$ can be recovered from~$F$
(Corollaries~\ref{c:ungradedbetti} and~\ref{c:gradedbetti}).

\begin{prop}\label{p:freeresextravars}
A minimal free resolution of $J_S$ can be obtained as the tensor
product of a minimal free resolution of $I_S$ with a Koszul complex.
\end{prop}
\begin{proof}
Non-minimality of $m, a_1, \ldots, a_{m-1}$ as generators for $S$ is
reflected in $J_S$ by binomial generators without $y$.  More
specifically, if $a_i + a_j = a_{i+j}$, then $b_{i,j}=0$ and the
binomial $x_i x_j - x_{i+j}$ appears in $J_S$.  Let $\mathcal{A}(S) =
\{m,a_{i_1}, a_{i_2}, \ldots, a_{i_r}\}$ be the elements $a_i$ that
minimally generate $S$.  Though $I_S$ naturally lives in
$\kk[y,x_{i_1},x_{i_2}, \ldots, x_{i_r}]$, consider it as an ideal in
$R$ via the natural inclusion map.  For each nonzero $w \in \Ap(S) -
\mathcal{A}(S)$, pick one of the binomials $f_w = x_w - x_u x_v$.
These binomials form a regular sequence on $R$, so the ideal $I_W$
generated by the $f_w$ is resolved by a Koszul
complex~$\mathcal{K}_\bullet$.  Writing $\mathcal{G}_\bullet$ for a
minimal free resolution of~$I_S$, the only nontrivial homology of
$\mathcal{G}_\bullet \otimes_R \mathcal{K}_\bullet$ occurs in
homological degree~$0$ and is isomorphic to $H_0(\mathcal{G}_\bullet)
\otimes H_0(\mathcal{K}_\bullet) = R/I_S \otimes_R R/I_W = R/J_S$,
where the last equality is because the $f_w$ form a regular sequence
over $R/I_S$.  Therefore $\mathcal{G}_\bullet \otimes
\mathcal{K}_\bullet$ is a minimal free resolution of $R/J_S$.
\end{proof}

\begin{example}\label{ex:tensorcomplexm4}
The underlying structure as a tensor of two resolutions is readily
seen in Figure~\ref{fig:specialization}, which resolves $J_S$ for
$\Ap(S)=\{4,a_1,2a_1,a_3\}$.  This example was obtained by computing
the Ap\'ery resolution for $J_S$ and then trimming away any constant
entries using row and column operations as described in
Theorem~\ref{t:specializationindependent}.
\end{example}

We include the proof of the following, despite its appearance in Theorem~\ref{t:kunzbetti} as recovered from~\cite{kunz}, to demonstrate how the Ap\'ery resolution maps in Theorem~\ref{t:medresolution} lend themselves to specialization to the faces of $C_m$, as well as to contrast its content with that of Theorem~\ref{t:specializationindependent}.  

\begin{cor}\label{c:ungradedbetti}
Let $S$ and $T$ be numerical semigroups corresponding to points
interior to the same face~$F$ of the Kunz cone~$\mathcal{C}_m$.  The
Ap\'ery ideals of $S$ and $T$ share the same Betti numbers, as do the
defining toric ideals of $S$ and $T$.  In particular, $\beta_d(J_S) =
\beta_d(J_T)$ and $\beta_d(I_S) = \beta_d(I_T)$ for all $d \geq 0$.
\end{cor}

\begin{proof}
Let $\mathcal F_\bullet$ and $\mathcal F_\bullet'$ be the Ap\'ery
resolutions of $J_S$ and $J_T$ respectively.  In the case that $S$ and
$T$ are both MED, so the face~$F$ is the entirety of~$C_m$, both
resolutions are minimal by Corollary~\ref{c:minimal-for-med} and have
the same modules at each homological degree, so $\beta_d(J_S) =
\beta_d(J_T)$ holds immediately.

If $\mathcal F_\bullet$ and $\mathcal F_\bullet'$ are not minimal,
then the resolutions have $\pm 1$ entries in identical places in their
resolutions, once again because $S$ and $T$ lie interior to the same
face~$F$ and thus have the same $b_{i,j} = 0$, meaning that the same
entries $\pm y^{b_{i,j}}$ become $\pm 1$.  Because the Betti numbers
of any positively graded ideal~$I$ equal the dimensions of the graded
vector spaces $\Tor_\bullet(I,\kk)$, consider $\mathcal F_\bullet
\otimes_\kk \kk$ and $\mathcal F_\bullet' \otimes_\kk \kk$.  The
differentials in these complexes are identical: they are matrices of
$0$s and $\pm 1$s with units in matching places.  Therefore, their
kernels and images are the same at each homological degree,~so
$$
  \beta_d(J_S) = \dim \Tor_d^\kk(J_S,\kk) = \dim
  \Tor_d^\kk(J_T,\kk) = \beta_d(J_T).
$$

Next consider $I_S$ and $I_T$.  By
Proposition~\ref{p:freeresextravars}, $\mathcal F_\bullet = (\mathcal
G_\bullet \otimes \mathcal K_\bullet)$ and $\mathcal F_\bullet' =
(\mathcal G_\bullet' \otimes \mathcal K_\bullet)$, where $\mathcal
G_\bullet$ and $\mathcal G_\bullet'$ are minimal free resolutions of
$I_S$ and $I_T$, respectively, and $\mathcal K_\bullet$ is the Koszul
resolution on the extraneous binomials.  Tensoring with $\mathcal
K_\bullet$ exerts the same invertible change on the Betti numbers of
$\mathcal G_\bullet$ and $\mathcal G_\bullet'$.  More specifically,
let
$$
  g_S(t) = \sum\limits_{i=0}^p \beta_i(I_S) t^i
  \qquad \text{and} \qquad
  f_S(t) = \sum\limits_{i=0}^q \beta_i(J_S) t^i
$$
be the generating functions
for the Betti numbers of $\mathcal G_\bullet$ and $\mathcal F_\bullet$
respectively.  Since $\mathcal K_\bullet$ is a Koszul resolution of $r
= m - |\mathcal A(S)|$ elements, $f_S(t) = (1+t)^r g_S(t)$.  Thus,
$$
  (1+t)^r g_S(t) = f_S(t) = f_T(t) = (1+t)^r g_T(t)
$$
$g_S(t) = g_T(t)$, meaning $\beta_i(I_S) = \beta_i (I_T)$ for all $i
\geq 0$.
\end{proof}

\begin{thm}\label{t:specializationindependent}
Consider the set 
$$
  \mathcal M
  =
  \{x_i : 1 \le i \le m-1\} \cup \{y^{b_{i,j}} : 1 \le i, j \le m - 1\}
$$
of formal symbols appearing as matrix entries in
Ap\'ery resolutions.  (Lemma~\ref{l:quotient} ensures every nonzero
matrix entry is accounted for in~$\mathcal M$).  Fix a face $F$ of
$C_m$.  There is a sequence of matrices, whose entries are
$\kk$-linear combinations of formal products of elements of $\mathcal
M$, with the following property: for each numerical semigroup~$S$
indexed by the relative interior of~$F$, substituting $R$-variables
and the values $b_{i,j}$ for $S$ into the entries of each matrix
yields boundary maps for a graded minimal free resolution of~$J_S$.
\end{thm}
\begin{proof}
Fix a numerical semigroup $S$ with multiplicity $m$.  Let
$$
  \mathcal N_S = \{ x_i : 1 \le i \le m-1 \} \cup \{ y^{b_{i,j}} : 1
  \le i, j \le m - 1 \text{ and } b_{i,j} > 0 \} \subseteq \mathcal M
$$
denote the set of elements of $\mathcal M$ corresponding to
positive-degree monomials in~$R$ under the grading by $S$.  If $S$ is
MED, then $\mathcal M = \mathcal N_S$; otherwise they are distinct.

By \cite[Theorem~20.2]{Eis95} (see also \cite[Exercises~1.10
and~1.11]{cca}), the matrices in any free resolution for $J_S$ can,
via a sequence of row and column operations that preserve homogeneity,
be turned into block diagonal matrices with $2$ blocks:\ (i)~a matrix
with no nonzero constant entries and at least one nonzero entry in
each row and column; and (ii)~a matrix with no nonconstant entries and
at most one nonzero entry in each row and column.  After doing this,
restricting to each block~(i) yields a minimal free resolution for
$J_S$.

One way to select the aforementioned row and column operations is as
follows.  Begin with the matrices $M_i$ for the maps $\partial_i$ for
the Ap\'ery resolution, and perform the following for each $i = 1, 2,
\ldots, m-1$, assuming that, as a result of prior operations, any
column of $M_i$ with a nonzero constant entry has no other nonzero
entries.

\begin{itemize}
\item%
First use nonzero constant entries of $M_i$ to clear all other entries
in their respective rows.  If $i=1$, then no such rows exist.  Fix a
row $R$ of $M_i$ with a nonzero constant entry $c$, say in column
$C_1$ with corresponding row $R_1$ in the matrix $M_{i+1}$.  For each
nonzero entry $f$ in $R$, say in a column $C_2 \ne C_1$ with
corresponding row $R_2$ in $M_{i+1}$, subtract $c^{-1} f \cdot C_1$
from~$C_2$ and add $c^{-1} f \cdot R_2$ to~$R_1$.  Once this is done,
$c$ will be the only nonzero entry in $R$, and in fact $c$ will be the
only nonzero entry in row $R$ and column $C_1$.  Moreover, since
$M_iM_{i+1} = 0$, the row~$R_1$ in $M_{i+1}$ only has entries~$0$.

\item%
Next use nonzero constant entries of $M_{i+1}$ to clear all other
entries in their respective columns. Fix a column $C$ of $M_{i+1}$
with a nonzero constant entry $c$, say in row $R_1$ with corresponding
column $C_1$ in the matrix $M_i$.  For each nonzero entry $f$ in $C$,
say in a row $R_2 \ne R_1$ with corresponding column $C_2$ in $M_i$,
subtract $c^{-1} f \cdot R_1$ from $R_2$ and add $c^{-1} f \cdot C_2$
to $C_1$.  Once this is done, $c$ will be the only nonzero entry in
column $C$, so since $M_iM_{i+1} = 0$, the column $C_1$ now only has
entries~$0$.  Moreover, all changes to $M_i$ only affect the (now all
0) column $C_1$, so $M_i$ still has the property that every nonzero
constant entry is the only nonzero entry in its row and column.


\end{itemize}
Once the above operations are completed for each $i$, the rows and
columns may be permuted to obtain the desired blocks.

The key observation is that in the above sequence of row and column
operations, the entry $f$ is an existing matrix entry.  As such, after
each row or column operation, every matrix entry $g$ can be written as
a $\kk$-linear combination of products of (possibly constant) elements
of $\mathcal M$.  Moreover, if $g$ is a nonzero constant, then $g$ has
degree $0$ under the grading by $S$, so by homogeneity, the
aforementioned expression for $g$ cannot contain any monomials in
$\mathcal N_S$, since it must be a $\kk$-linear combination of
products of degree-$0$ elements of~$\mathcal M$.

Now, fix a numerical semigroup $T$ in the same face $F$ of the Kunz
cone $C_m$ as $S$.  The sets $\mathcal M \setminus \mathcal N_S$ and
$\mathcal M \setminus \mathcal N_T$ each contain $y^{b_{i,j}}$
whenever $F$ is contained in the facet $z_i + z_j = z_{i+j}$, and thus
$\mathcal N_S = \mathcal N_T$.  As a consequence of the preceding
paragraph, applying an identical sequence of row and column operations
to the Ap\'ery resolution for $J_T$ yields nonzero constant entries in
precisely the same locations at each step of the process.
This completes the proof.
\end{proof}

Theorem~\ref{t:specializationindependent} yields the following graded
refinement of Corollary~\ref{c:ungradedbetti}.

\begin{cor}\label{c:gradedbetti}
For each $i \in \{0, \ldots, m-1\}$, writing $[i]_m = i + m\ZZ$,
$$
  \sum_{b \in [i]_m} \beta_{d,b}(J_S) = \sum_{b \in [i]_m}
  \beta_{d,b}(J_{S'}).
$$
The same relationship holds between the Betti numbers of $I_S$ and
$I_{S'}$.
\end{cor}
\begin{proof}
Apply Theorem~\ref{t:specializationindependent} for the first claim,
and subsequently apply Proposition~\ref{p:freeresextravars} for the
final claim.
\end{proof}

\begin{figure}[t]\label{fig:specialization}
{\footnotesize
$
\begin{array}{c@{\:}c@{\:}c@{\:}c@{\:}c}
&
\begin{blockarray}{rcccc}
  &
  \s \textbf{1,1} &
  \s \textbf{3,3} &
  \s \textbf{2,1} &
  \s \textbf{3,1} \\
  \begin{block}{r@{\,\,}[@{\,\,}l@{\,}|*{3}{@{\,\,}l}]}
    \s \textbf{000} & 
    x_1^2  - x_2 &
    x_3^2  - x_1^2 y^{b_{33}} &
    x_1^3 - x_3 y^{b_{12}} &
    x_1x_3 - y^{b_{13}} \\
  \end{block}
\end{blockarray}
&
\\[-1em]
0 \leftarrow R & \filleftmap &
\end{array}
$

$
\begin{array}{c@{\:}c@{\:}c@{\:}c@{\:}c@{}c}
&
\begin{blockarray}{@{}r*{5}{@{\,\,}c}}
  \\
  &
  \s \textbf{3,23} &
  \s \textbf{2,12} &
  \s \textbf{2,13} &
  \s \textbf{3,13} &
  \s \textbf{3,12} \\
  \begin{block}{@{}l@{\,\,}[*{3}{@{\,\,}l}|*{2}{@{\,\,}l}@{\,}]}
    \s \textbf{1,1} & \phm x_3^2 - x_1^2 y^{b_{33}} & \phm x_1^3 - x_3 y^{b_{12}} & \phm x_1x_3 - y^{b_{13}} & \phm                        & \phm                      \\
    \cline{2-6}
    \s \textbf{3,3} & \rlm (x_1^2 - x_2)            & \phm                        & \phm                     & \phm x_1                    & \rlm y^{b_{12}}           \\
    \s \textbf{2,1} & \phm                          & \rlm (x_1^2 - x_2)          & \phm                     & \phm y^{b_{33}}             & \rlm x_3                  \\
    \s \textbf{3,1} & \phm                          & \phm                        & \rlm (x_1^2 - x_2)       & \rlm x_3                    & \phm x_1^2                \\
  \end{block}
\end{blockarray}
&&
\begin{blockarray}{@{}r*{3}{@{\,\,}c}}
  &
  \s \textbf{3,[3]} &
  \s \textbf{2,[3]} \\
  \begin{block}{@{}l@{\,\,}[*{3}{@{\,\,}l}@{\,}]}
    \s \textbf{3,23} & \phm x_1         & \rlm y^{b_{12}}  \\
    \s \textbf{2,12} & \phm y^{b_{33}}  & \rlm x_3         \\
    \s \textbf{2,13} & \rlm x_3         & \phm x_1^2       \\
    \cline{2-3}
    \s \textbf{3,13} & \phm x_1^2 - x_2 & \phm             \\
    \s \textbf{3,12} & \phm             & \phm x_1^2 - x_2 \\
  \end{block}
\end{blockarray}
&
\\[-1em]
R^4 & \filleftmap & R^5 & \filleftmap & R^2 & \leftarrow 0
\end{array}
$
\caption{A specialization of the $m = 4$ Ap\'ery resolution where
$b_{11} = 0$, so $a_2 = 2a_1$.  Note this forces $b_{13} = b_{23}$.}
\label{fig:m4spec}
}
\end{figure}

\section{Open questions}\label{sec:openquesitons}

Several of the open questions presented here relate to the defining
toric ideal~$I_S$.  One of the main results of~\cite{kunzfaces1}
identifies a finite poset corresponding to each face~$F$ of~$C_m$,
called the \emph{Kunz poset} of~$F$.  If a point interior to~$F$
indexes a numerical semigroup~$S$, then this poset coincides with the
divisibility poset of~$\Ap(S)$.  In~\cite{kunzfaces3}, the Kunz poset
of $F$ is used to obtain a parametrized family of minimal binomial
generating sets, one for the defining toric ideal $I_S$ of each
numerical semigroup $S$ in $F$.  The last three binomials in the first
matrix in Figure~\ref{fig:specialization} constitute one such example
for the relevant facet $F$ of $C_4$.  This provides a natural
candidate for the first matrix in the resolution conjectured as
follows.

\begin{conj}\label{conj:parametrizedfreeres}
For each face $F$ of $C_m$, there exists a parametrized family of
minimal resolutions, one for the defining toric ideal $I_S$ of each
numerical semigroup $S$ indexed by the interior of~$F$, akin to those
obtained in Theorem~\ref{t:specializationindependent} for Ap\'ery
ideals.
\end{conj}

Unlike the proof of Theorem~\ref{t:specializationindependent}, the
proofs in~\cite{kunzfaces3} are constructive, utilizing the Kunz poset
structure of the face $F$ containing $S$.  Intuitively, the set of
factorizations of elements of $\Ap(S)$ (a set which depends only on
$F$) forms the staircase of a monomial ideal $M$.  If each element of
$\Ap(S)$ factors uniquely, then $I_S$ has exactly one binomial
generator for each minimal monomial generator of $M$.  If some
elements of $\Ap(S)$ have multiple factorizations, then
graph-theoretic methods can be used to partition some of the minimal
generators of $M$ into blocks (called \emph{outer Betti elements}) and
construct one minimal binomial generator of $I_S$ for each block.  We
refer the reader to~\cite[Section~5]{kunzfaces3} for the full
construction; additional examples can be found in~\cite{minprescard}.

The non-constructive nature of the proof of
Theorem~\ref{t:specializationindependent}, along with the constructive
nature of the proofs in~\cite{kunzfaces3}, motivates the following.


\begin{prob}\label{prob:explicitconstruction}
Find an explicit combinatorial (e.g., poset-theoretic) construction of
the matrices obtained in Theorem~\ref{t:specializationindependent} and
conjectured in Conjecture~\ref{conj:parametrizedfreeres}.
\end{prob}



There is a long history of topological formulas for Betti numbers of
graded ideals (e.g., Hochster's formula for squarefree monomial
ideals~\cite{Hoc77} (see \cite[Chapter~1]{cca}), squarefree divisor
complexes for toric ideals~\cite{squarefreedivisorcomplex}, or the use
of poset homology for computing Poincar\'e series of semigroup
algebras~\cite{shellmonoid}).  The following is thus a natural
problem.

\begin{prob}\label{prob:topologicalbetti}
Given a face~$F$ of~$C_m$, find a topological formula for extracting
the value in the equation in Corollary~\ref{c:gradedbetti} from the
Kunz poset of~$F$.
\end{prob}

As mentioned in Example~\ref{ex:kunzm4}, the ray $(1,2,1)$ of $C_4$
contains positive integer points, but none correspond to a numerical
semigroup under Proposition~\ref{prop:kunzinteriormed}.  Indeed, the
first and third coordinates of any such point must be equal, but any
point $(a_1, a_2, a_3) \in C_4$ corresponding to a numerical semigroup
must have $a_i \equiv i \bmod 4$ for each $i$.  However, the
construction in~\cite{kunzfaces1} still associates a poset to this
ray, and naively following the construction in~\cite{kunzfaces3} for a
binomial generating set with $y = 0$ yields the artinian binomial
ideal $\<x_1^2 - x_3^2, x_1^3, x_1x_3, x_3^3\>$.  This motivates the
following.

\begin{prob}\label{prob:negativefaces}
Extend the correspondance in Proposition~\ref{prop:kunzinteriormed} to
a family of lattice ideals that includes the defining toric ideals of
numerical semigroups but reaches points in faces of $C_m$ that do not
contain numerical semigroups.
\end{prob}



\section*{Acknowledgements}

The authors would like to thank David Eisenbud for identifying portions of Kunz's work in~\cite{kunz} that had been overlooked all these years, and Frank-Olaf Schreyer for informing us about the content of Remark~\ref{r:schreyerresolution}.


\end{document}